\documentclass[a4paper]{article}
\usepackage{cite}
\usepackage[english]{babel}
\usepackage[utf8x]{inputenc}
\usepackage[T1]{fontenc}
\usepackage{pgf,tikz}
\usepackage{amsfonts}
\usetikzlibrary{arrows}
\usepackage[bottom]{footmisc}
\usepackage{appendix}
\usepackage[a4paper,top=3cm,bottom=3cm,left=3cm,right=3cm,marginparwidth=2cm]{geometry}

\usepackage{amsthm}
\newtheorem{thm}{Theorem}
\newtheorem{cor}{Corollary}
\usepackage{amsmath}
\usepackage{graphicx}
\usepackage[colorinlistoftodos]{todonotes}
\usepackage[colorlinks=true, allcolors=blue]{hyperref}
\usepackage{pgfplots} 
\usepackage{authblk}
\pgfplotsset{grid style={gridColor,line width=0pt}}

\newcommand\ep{\varepsilon}

\usepackage{color}

\title{Long-time behaviour of a model for p62-ubiquitin aggregation in cellular autophagy}
\author[1]{Julia Delacour}
\author[2]{Christian Schmeiser}
\author[3]{Peter Szmolyan} 
\affil[1]{Sorbonne Universit\'e, Inria, Universit\'e Paris-Diderot, CNRS, Laboratoire Jacques-Louis Lions, 75005 Paris, France}
\affil[2]{Faculty of Mathematics, University of Vienna, Oskar-Morgenstern-Platz 1, 1090 Wien, Austria}
\affil[3]{Institut f\"ur Angewandte und Numerische Mathematik, TU Wien, Wiedner Hauptstr. 8--10, 1040 Wien, Austria}
\begin{document}
\maketitle

\begin{abstract}
The qualitative behavior of a recently formulated ODE model for the dynamics of heterogenous aggregates is analyzed. Aggregates contain two types of
particles, oligomers and cross-linkers. The motivation is a preparatory step of cellular autophagy, the aggregation of oligomers of the protein p62 in the
presence of ubiquitin cross-linkers. A combination of explicit computations, formal asymptotics, and numerical simulations has led to conjectures on the 
bifurcation behavior, certain aspects of which are proven rigorously in this work. In particular, the stability of the zero state, where the model has a smoothness
deficit is analyzed by a combination of regularizing transformations and blow-up techniques. On the other hand, in a different parameter regime, the existence
of polynomially growing solutions is shown by Poincar\'e compactification, combined with a singular perturbation analysis . 
\end{abstract}

\section{Introduction}
A preparatory step of cellular autophagy is the aggregation of cellular waste material before inclusion in an autophagosome and, later, a lysosome, which 
are vesicular structures, where the waste is eventually decomposed. In vitro studies of the evolution of heterogeneous aggregates of the proteins p62 and 
ubiquitin \cite{Martens} have motivated the formulation of a mathematical model of this process \cite{firstarticle}. The model has the form of an ODE system,
which shows different qualitative behaviour in three different regions of parameter space. This statement is based on formal asymptotics and numerical
simulations carried out in \cite{firstarticle}. Since some of these observations are not accessible to standard dynamical systems methods, it is the purpose
of this work to provide a rigorous analysis.

The model is based on the assumption that ubiquitinated waste material serves as a cross-linker between p62 oligomers of a fixed size $n\ge 3$, where 
each monomer can serve as a binding site for cross-linking.
It describes the evolution of the size of aggregates through the evolution of three parameters $p$, $q$ and $r$, where $p$ represents the number of one-hand bound ubiquitin links in the aggregate (in green in  Figs. \ref{fig:1} and \ref{fig:2}), $q$ represents the number of both-hand bound cross-links 
(in red in Figs. \ref{fig:1} and \ref{fig:2}), and $r$ represents the number of p62$_n$ oligomers in the aggregate (in black in Figs. \ref{fig:1} and \ref{fig:2} with 
$n=5$). 
Since we are interested in large aggregates, the variables $(p,q,r)$ are considered as continuous after an appropriate scaling. The state space is a subset 
of the positive octant of $\mathbb{R}^3$ determined by two constraints: For a connected aggregate the number of two-hand bound cross-links has to be
at least the number of p62 oligomers minus one. In the continuous description this becomes the constraint $q\ge r$. Since the total number of binding sites
on the p62 oligomers in an aggregate is $nr$, the number of free binding sites is equal to $nr-p-2q$, which has to be nonnegative.
The dynamics of an
aggregate is governed by the basic binding and unbinding reactions between cross-linkers and p62 oligomers. Since the reaction rates depend on the state
of the reaction partners and of the aggregate, six different reactions have been considered in \cite{firstarticle}
(see  Figs. \ref{fig:1} and \ref{fig:2}). The models for the reaction rates are based on the law of mass action. However, since the shape of an aggregate is
not described unambiguously by the parameters $(p,q,r)$, some additional empirical modeling assumptions are required.

\begin{figure}[h!] \centering
\includegraphics[width=0.7\textwidth]{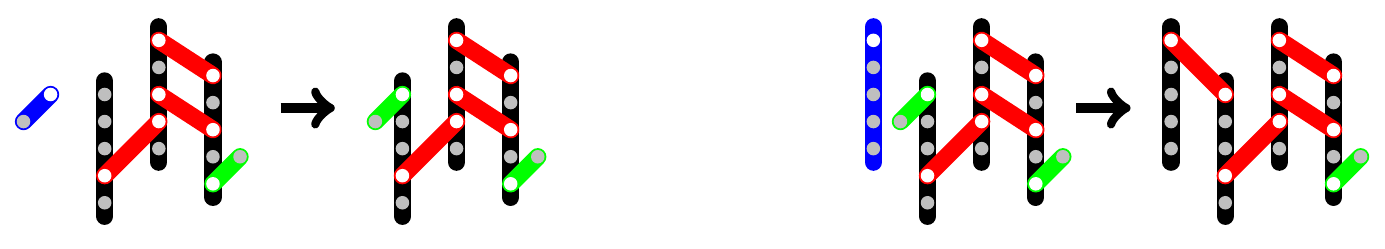}
\label{fig:1}
\caption{Left: Reaction 1: addition of a free cross-linker to the aggregate. Right: Reaction 2: addition of a p62$_5$ oligomer to the aggregate.}
\end{figure}

\begin{figure}[h!] \centering
\includegraphics[width=0.7\textwidth]{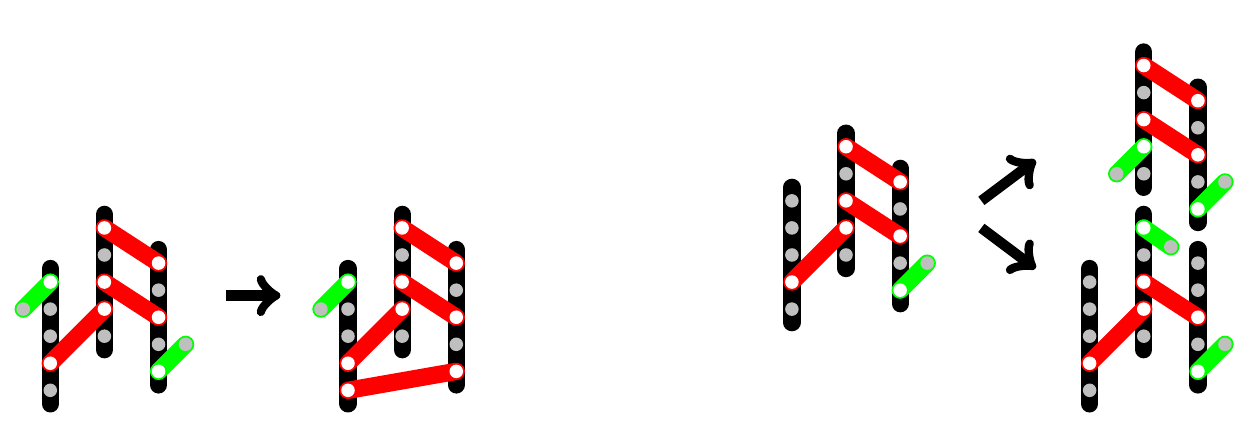}
\caption{Left: Reaction 3 is a rearrangement making the aggregate more cohesive. Right: the breaking of a cross-link bound can lead either to the reverse of Reaction 2 of of Reaction 3.}
\label{fig:2}
\end{figure}

\emph{Reaction 1} is the addition of a free cross-linker to the aggregate. This is a second order reaction with a rate proportional to the concentration
of free cross-linkers and to the number of free binding sites on oligomers. Since the supply of free cross-linkers and oligomers has been modeled as not
limiting in \cite{firstarticle}, the reaction rate is written as $\kappa_1(nr-p-2q)$ with rate constant $\kappa_1$, which can be seen as proportional to the
cross-linker concentration, modeled as constant. Similarly \emph{Reaction 2,} the addition of a free oligomer to the aggregate, is modeled as a first order 
reaction with rate $\kappa_2 p$ proportional to the number of one-hand bound cross-linkers. \emph{Reaction 3} is consolidating the aggregate by 
building an additional cross-link using a so far only one-hand bound cross-linker. Its rate is $\kappa_3 p(nr-p-2q)$. For the reverse of Reaction 1, the rate 
$\kappa_{-1}p$ should not be a surprise. The reverses of Reactions 2 and 3 are actually the same reaction with rate $\kappa_- q$, but with possibly different outcomes. Therefore we write their rate constants as $\kappa_{-2}:= \kappa_- \alpha$ and $\kappa_{-3}:= \kappa_- (1-\alpha)$, with $\alpha=\alpha(q,r)
\in [0,1]$. The reverse of Reaction 2, i.e. loss of an oligomer, never happens in a fully connected aggregate with $nr=2q$. It always happens in a minimally
connected aggregate with $q=r$. This motivates the choice $\alpha = \frac{nr-2q}{(n-2)r}$. Concerning the outcome of this reverse reaction, it has to be taken 
into account that the loss of an oligomer might also mean a loss of one-hand bound cross-links attached to it. This produces the loss term 
$\kappa_- q \frac{(n-1)p}{(n-2)r}$ (see \cite{firstarticle} for details). It is now straightforward to write down the ODE problem governing the evolution of the
state variables:
\begin{equation} \label{eq_1}
\begin{aligned}
  \dot p &= (\kappa_1 - \kappa_3 p) (nr-p-2q) + \kappa_{-}q \left( 1 - \frac{(n-1)p}{(n-2)r}\right) - (\kappa_2 + \kappa_{-1})p \,, &\qquad p(0)=p_0\,,\\
  \dot q &= \kappa_2 p + \kappa_3 p(nr-p-2q) - \kappa_{-}q\,, &\qquad q(0)=q_0\,,  \\
  \dot r &= \kappa_2 p - \kappa_{-} q \alpha(q,r), \qquad\alpha(q,r) = \frac{nr -2q}{(n-2) r}\,, &\qquad r(0)=r_0\,,
  \end{aligned}
 \end{equation}
 with the inequalities 
 \begin{equation} \label{eq_2}
 nr - p - 2q \geq 0 \,, \qquad  q\ge r \,,
 \end{equation}
 implying
 $$
 0 \leq \alpha(q,r) \leq 1 \,. 
 $$
 We recall from \cite[Theorem 1]{firstarticle} that for initial data $p_0,q_0,r_0>0$ satisfying \eqref{eq_2}, which we assume in the following, the initial value
 problem \eqref{eq_1} has a unique, global solution propagating \eqref{eq_2}, the nonnegativity of the components, and in particular
\begin{equation}\label{r-pos}
   r(t),q(t) > 0 \,,\qquad t\ge 0 \,.
\end{equation}

 \begin{figure}[h!] \centering
\includegraphics[width=0.5\textwidth]{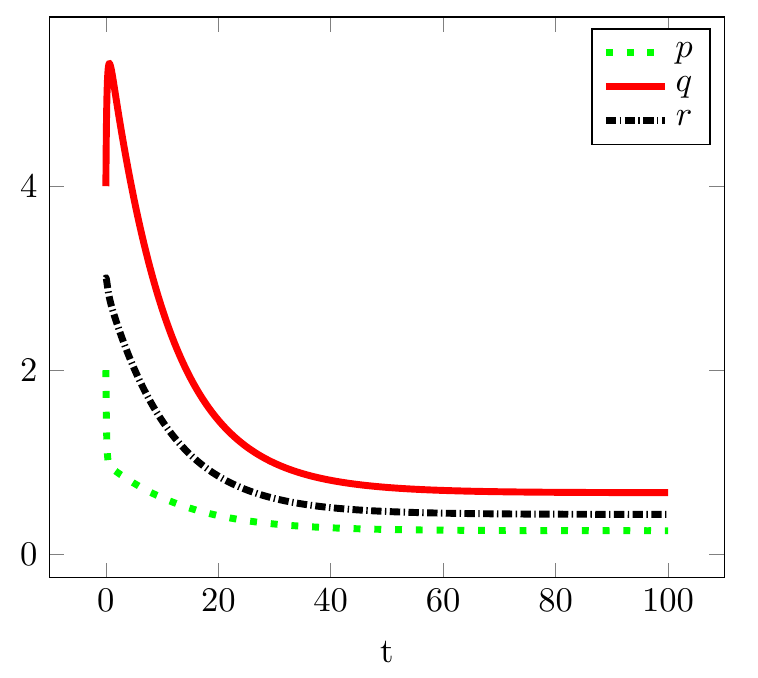}
\caption{Evolution of a state $(p,q,r)$ of initial size $(2,4,3)$ with parameters $n=5$, 
$\kappa_1= \kappa_2=\kappa_3=\kappa_{-1}=1$, and $\kappa_{-}=0.6$, giving $0<\bar\alpha < 1$ and convergence to the nontrivial equilibrium.}
\label{fig:3} 
\end{figure}

\begin{figure}[h!] \centering
\includegraphics[width=0.5\textwidth]{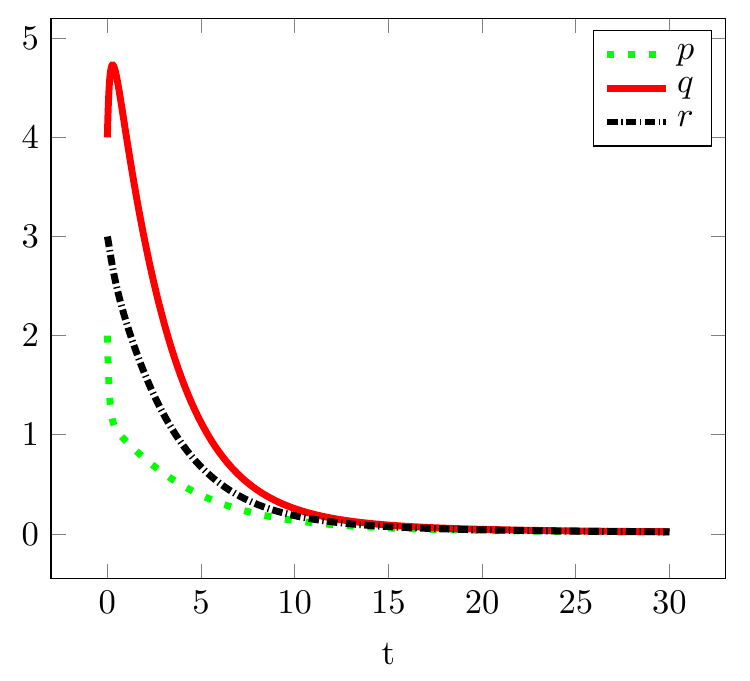}
\caption{Evolution of a state $(p,q,r)$ of initial size $(2,4,3)$ with parameters $n=5$, $\kappa_1= \kappa_2=\kappa_3=\kappa_{-1} = 1$ ,
and $\kappa_{-}=0.93$, giving $\bar\alpha>1$ and convergence to the zero steady state.}
\label{fig:5} 
\end{figure}

\begin{figure}[h!] \centering
\includegraphics[width=0.5\textwidth]{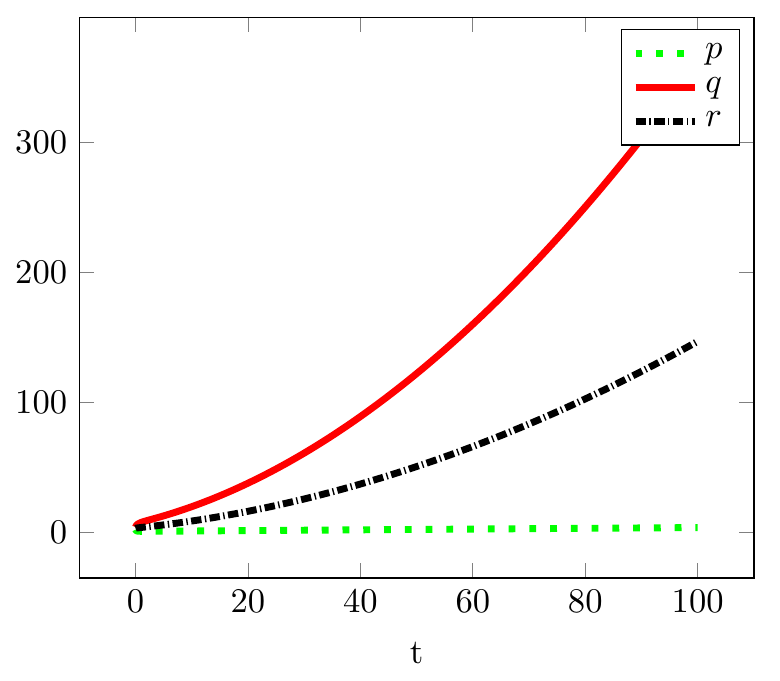}
\caption{Evolution of a state $(p,q,r)$ of initial size $(2,4,3)$ with parameters $n=5$,
$\kappa_1= \kappa_2=\kappa_3=\kappa_{-1} = 1$, and $\kappa_{-}=0.2$, giving $\bar\alpha<0$ and a polynomially growing aggregate.}
\label{fig:4} 
\end{figure}
 
The search for steady states \cite{firstarticle} has suggested a splitting of the parameter space into three regions. Besides the trivial steady state
$(p,q,r)=(0,0,0)$, only one other equilibrium may exist, which can be computed explicitly:
\begin{eqnarray}
& \bar p =\frac{\kappa_- A}{\kappa_2}  \frac{1-\bar\alpha}{\bar\alpha} \,,\qquad \bar q = A \frac{1-\bar\alpha}{\bar\alpha^2} \,, \qquad
  \bar r = \frac{2A}{n-(n-2)\bar\alpha} \frac{1-\bar\alpha}{\bar\alpha^2}\,, \label{ss}\\
&\mbox{with}\quad\bar \alpha = \frac{n}{n-2} + \frac{\kappa_{-1} + \kappa_1 - \sqrt{(\kappa_1 + \kappa_{-1})^2 + 4  \kappa_1 \kappa_2 (n-1)}}{\kappa_{-} (n-1)} \,, \quad A = \frac{2\kappa_1 \kappa_2^2 (n-2)}{\kappa_3 \kappa_- (\kappa_-(n-1)(n - (n-2)\bar\alpha) + 2\kappa_{-1}(n-2))} \,.\nonumber
\end{eqnarray}
Since $\bar\alpha = \alpha(\bar q,\bar r)$ is the equilibrium value of $\alpha$, the nontrivial steady state is relevant only in the parameter region defined
by $0<\bar\alpha<1$. It has been conjectured in \cite{firstarticle} that in this parameter region $(\bar p,\bar q, \bar r)$ is globally attracting, which has been
supported by numerical simulations (see also Fig. \ref{fig:3}). Local stability could in principle be examined by linearization. However, the complexity of the
resulting formulas has been prohibitive.
 
Since $(\bar p,\bar q,\bar r)\to (0,0,0)$ as $\bar\alpha\to 1-$, it seems natural to expect a transcritical bifurcation at $\bar\alpha=1$ with stability of
the trivial steady state for $\bar\alpha>1$. Again the conjecture of global asymptotic stability of $(0,0,0)$ for $\bar\alpha>1$ has been supported by
simulations (see for example Fig. \ref{fig:5}). The right hand sides of \eqref{eq_1} are continuous up to the origin (when considered as an element of the set 
of admissible states), since $0\le\alpha(q,r)\le 1$ and $p/r\le n$. However, their nonsmoothness prohibits a standard local stability or bifurcation analysis. 
The expected local stability behaviour (asymptotic stability for $\bar\alpha>1$, instability for $\bar\alpha<1$) is proven in Section \ref{sec:zero}. The analysis 
is based on a regularizing transformation, which makes the steady state very degenerate, combined with a blow-up analysis \cite{dumortier}.

The fact that the components of the nontrivial equilibrium tend to infinity when $\bar\alpha\to 0+$ suggests that solutions might be unbounded
for $\bar\alpha<0$. In this parameter region approximate solutions with polynomial growth of the form 
\begin{equation}\label{poly-ansatz}
  p(t) = p_1 t+ \textit{o}(t) \,,\qquad q(t) = q_2 t^2 + \textit{o}(t^2) \,,\qquad r(t) = \frac{2q_2}{n} t^2+ \textit{o}(t^2) \,,\qquad \mbox{as } t\to\infty \,,
\end{equation}
have been constructed in \cite{firstarticle} by formal asymptotic methods. It has also been shown that no
other growth behaviour (polynomial with other powers or exponential) should be expected, and the conjecture that all solutions have the constructed 
asymptotic behaviour is again verified by simulations (see Fig. for example \ref{fig:4}). We justify the formal asymptotics in Section \ref{sec:poly}. A variant 
of Poincar\'e compactification \cite{Poincare} produces a problem with bounded solutions and with three different time scales, which is analyzed by singular 
perturbation methods \cite{Fenichel}.
The final result is existence and semi-local stability of the polynomially growing solutions, where 'semi-local' means that initial data have to be large
with relative sizes as in \eqref{poly-ansatz}.

The article is concluded by a discussion section about biological interpretation of our results as well as perspectives. 

\section{Local stability of the zero steady state}\label{sec:zero}

In this section, we study under which conditions small aggregates tend to disaggregate. This is equivalent to studying the stability of the zero-steady-state 
$(p,q,r)=(0,0,0)$ of the system \eqref{eq_1}. Because of the appearance of the ratios $\frac{p}{r}$ and $\frac{q}{r}$, the Jacobian of the right hand side of
\eqref{eq_1} is not defined there. As a consequence of \eqref{r-pos} the regularizing transformation $\tau := \int_0^t r(s)^{-1} \mathrm{d}s$
is well defined and leads to 
\begin{equation} \label{eq_1-regularized}
\begin{aligned}
  \frac{dp}{d\tau} &= r(\kappa_1 - \kappa_3 p) (nr-p-2q) + \kappa_{-}q \left( r - \frac{(n-1)p}{n-2}\right) - (\kappa_2 + \kappa_{-1})pr \,, \\
  \frac{dq}{d\tau} &= \kappa_2 pr + \kappa_3 pr(nr-p-2q) - \kappa_{-}qr\,,  \\
  \frac{dr}{d\tau} &= \kappa_2 pr - \kappa_{-} q \frac{nr -2q}{n-2}\,.
  \end{aligned}
 \end{equation}
The regularization came at the expense that the zero steady state is degenerate in \eqref{eq_1-regularized}, since the right hand side is of second order in terms of the densities.
 A classical approach to study such non-hyperbolic points is \emph{blow-up} \cite{dumortier}. The standard blow-up transformation would be the
 introduction of spherical coordinates, blowing up the origin to the part of $\mathbb{S}^2$ in the positive octant. It is also common to work with \emph{charts}
 instead. In our case this preserves the polynomial form of the right hand side. Although the charts in the different coordinate directions are equivalent,
 since the state space is a subset of the positive octant, it has turned out to be convenient to use the $q$-chart, whence
the blow-up transformation $(p,q,r) \to (p_1,q_1,r_1)$ is given by 
\begin{equation}
  p= p_1q_1\,, \qquad q=q_1\,,\qquad r=r_1 q_1 \,,
\end{equation} 
and we also introduce another change of time scale: $T := \int_0^\tau q_1(\sigma)d\sigma$, again justified by \eqref{r-pos}, leading to
\begin{eqnarray} 
 \frac{dq_1}{dT} &=& q_1 r_1\left( \kappa_2 p_1 - \kappa_- \right) + \kappa_3 p_1 r_1 q_1^2(nr_1-p_1-2)\,,\nonumber\\
 \frac{dp_1}{dT} &=& r_1(\kappa_1 - \kappa_3 p_1 q_1)(nr_1-p_1-2) + \kappa_- \left(r_1 - \frac{n-1}{n-2}p_1\right) - (\kappa_2 + \kappa_{-1})p_1r_1
      - p_1 r_1(\kappa_2 p_1 - \kappa_-) \nonumber\\
      && -\kappa_3 p_1^2 r_1 q_1 (nr_1 - p_1 - 2) \,,\label{system-T}\\
 \frac{dr_1}{dT} &=& (1-r_1)\left( \kappa_2 p_1 r_1 + \kappa_- \left( \frac{2}{n-2} - r_1\right) \right) - \kappa_3 p_1 r_1^2 q_1 (nr_1 - p_1 - 2)\,.\nonumber
\end{eqnarray}
The invariant manifold $q_1=0$ of this system corresponds to the zero steady state of \eqref{eq_1}.
The inequalities \eqref{eq_2} become
$$
   r_1 \le 1 \,,\qquad 0\le p_1 \le nr_1-2 \,,
$$
in terms of the new variables, i.e. the dynamics of $(p_1,r_1)$ remains in the triangle depicted in Fig. \ref{fig:6}. Since $r_1\ge 2/n$, we conclude from the
equation for $q_1$ that the invariant manifold is locally exponentially attracting in the region to the left of the line $p_1 = \kappa_-/\kappa_2$. Since
$p_1\le n-2$, the inequality $\kappa_- > (n-2)\kappa_2$ already implies local asymptotic stability of the invariant manifold $q_1=0$ of \eqref{system-T}
and therefore of the zero steady state of \eqref{eq_1}. Note that $\kappa_- > (n-2)\kappa_2$ also implies $\bar\alpha > 1$ for $\bar\alpha$
defined by \eqref{ss}.

\begin{figure}[h!] \centering
\includegraphics[width=0.4\textwidth]{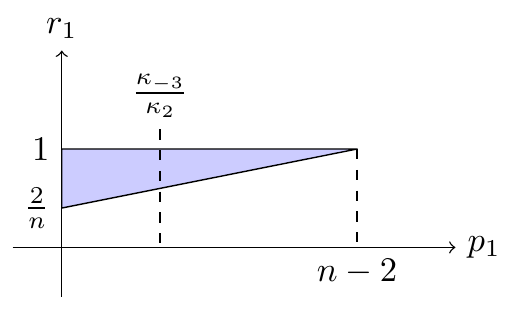}
\caption{The dynamics in the $(p_1,r_1)$-plane is limited to the shaded triangle because of the inequalities \eqref{eq_2}.}
\label{fig:6} 
\end{figure}

In the following we therefore consider the case $\kappa_- \le (n-2)\kappa_2$ (see Fig. \ref{fig:6}) and $\bar\alpha > 1$, where the latter is equivalent to
\begin{equation}\label{alpha-bigger-1}
  \kappa_1\kappa_2(n-2)^2 < \kappa_-(\kappa_1 + \kappa_{-1})(n-2) + \kappa_-^2(n-1) \,, 
\end{equation}
see also \cite[Equ. (26)]{firstarticle}. The flow on the invariant 
manifold $q_1=0$ of \eqref{system-T} is governed by the system
\begin{eqnarray} 
  \frac{dp_1}{dT} &=& r_1\kappa_1 (nr_1-p_1-2) + \kappa_- \left(r_1 - \frac{n-1}{n-2}p_1\right) - (\kappa_2 + \kappa_{-1})p_1r_1
      - p_1 r_1(\kappa_2 p_1 - \kappa_-) \nonumber\\
 \frac{dr_1}{dT} &=& (1-r_1)\left( \kappa_2 p_1 r_1 + \kappa_- \left( \frac{2}{n-2} - r_1\right) \right) \,.\label{q1=0}
\end{eqnarray}
In the right part of the triangle, i.e. for
$$
   r_1 \le 1 \,,\qquad \frac{\kappa_-}{\kappa_2}\le p_1 \le nr_1-2 \,,
$$
we have
\begin{eqnarray*}
   \frac{dp_1}{dT} &\le& r_1\kappa_1 \left( n - \frac{\kappa_-}{\kappa_2} - 2\right) + \kappa_- \left( r_1 - \frac{n-1}{n-2}\,\frac{\kappa_-}{\kappa_2}\right)
   - (\kappa_2 + \kappa_{-1})\frac{\kappa_-}{\kappa_2}r_1 \\
   &=& \frac{r_1\bigl(\kappa_1\kappa_2(n-2)^2 - \kappa_-(\kappa_1 + \kappa_{-1})(n-2)\bigr) - \kappa_-^2(n-1)}{\kappa_2(n-2)} 
   < \frac{(r_1-1)\kappa_-^2(n-1)}{\kappa_2(n-2)} \le 0 \,,
\end{eqnarray*}
where the strict inequality is due to \eqref{alpha-bigger-1}. This implies that all trajectories reach the left part of the triangle, i.e. $p_1 < \kappa_-/\kappa_2$
in finite time. 

By standard regular perturbation theory the dynamics for the full system \eqref{system-T}, when started close to the invariant manifold $q_1=0$, remains
close to the dynamics on the invariant manifold for finite time, until the region $p_1 < \kappa_-/\kappa_2$ is reached, where the invariant manifold is
attracting. Thus $q=q_1$ tends to zero and, by the inequalities \eqref{eq_2}, the same is true for $p$ and $r$.

Now we consider the case $\bar\alpha<1$, i.e. the opposite of inequality \eqref{alpha-bigger-1}, and look for a steady state on the invariant manifold 
$r_1=1$ of the system \eqref{q1=0}. Since
\begin{eqnarray*}
   \frac{dp_1}{dT}\Bigm|_{r_1=1,p_1=\kappa_-/\kappa_2} &=& \frac{\kappa_1\kappa_2(n-2)^2 - \kappa_-(\kappa_1 + \kappa_{-1})(n-2) - \kappa_-^2(n-1)}
       {\kappa_2(n-2)} >0 \,,\\
 \frac{dp_1}{dT}\Bigm|_{r_1=1,p_1=n-2} &=& -(n-2)(n\kappa_2 + \kappa_{-1}) < 0 \,,      
 \end{eqnarray*}
 there exists a steady state $(p_1,r_1) = (p_1^*,1)$ with $\kappa_-/\kappa_2 < p_1^* < n-2$, which is stable under the flow along $r_1=1$. On the other 
 hand
$$
   \frac{1}{1-r_1}\, \frac{dr_1}{dT}\Bigm|_{r_1=1,p_1=p_1^*} = \left( \kappa_2 p_1^* + \kappa_- \left( \frac{2}{n-2} - 1\right) \right) > \frac{2\kappa_-}{n-2} >0 \,,
$$
which implies stability of the manifold $r_1=1$ close to the steady state, and therefore stability of the steady state. The existence of a stable steady state 
on the invariant manifold $q_1=0$ of \eqref{system-T} in the region, where the manifold is repulsive, implies instability of the manifold and therefore also 
of the zero steady state of \eqref{eq_1}.
This completes the proof of the main result of this section.

\begin{thm}
Let $\bar\alpha$ be defined by \eqref{ss}. Then the steady state $(0,0,0)$ of the system \eqref{eq_1} is locally asymptotically stable for $\bar\alpha>1$
and unstable for $\bar\alpha<1$. 
\end{thm}

\section{Polynomially growing regime}\label{sec:poly}

The goal of this section is a rigorous justification of the formal asymptotics \eqref{poly-ansatz} (see \cite{firstarticle}) under the assumption 
$\bar\alpha<0$ with $\bar\alpha$ defined in \eqref{ss}, i.e.
\begin{equation}\label{alpha-smaller-0}
  4\kappa_1\kappa_2(n-2)^2 > n\kappa_-\bigl(2(\kappa_1 + \kappa_{-1})(n-2) + \kappa_-n(n-1) \bigr)\,, 
\end{equation}
see also \cite[Equ. (27)]{firstarticle}.

Considering \eqref{poly-ansatz}, it would be natural to write an equation for $p(t)/t$. It is easily seen from \eqref{eq_1} that its 
derivative contains terms of the order of $t^2$. Similarly the derivative of $q(t)/t^2$ has contributions up to the order of $t$, whereas the derivative of 
$r(t)/t^2$ is a combination of terms bounded as $t\to\infty$. This shows that we are confronted with a problem with different time scales, which will put us
into the realm of \emph{singular perturbation theory} (see, e.g. \cite{Fenichel,Szmolyan}). The leading order term in the fastest equation, i.e. the 
$p$-equation, is $-\kappa_3 p (nr-2q)$, from which it has been concluded in \cite{firstarticle} that $nr(t)\approx 2q(t)$ as $t\to\infty$. In a standard
singular perturbation setting, it should be possible to express $p(t)$ from this relation. Since this is not the case, our problem belongs to the family of
\emph{singular} singularly perturbed problems (see e.g. \cite{Schmeiser}) which, however, can be transformed to the standard regular form in many cases.

The introduction of $p(t)/t$, $q(t)/t^2$, $r(t)/t^2$, as new variables would lead to a study of bounded solutions, but to a non-autonomous system. 
We shall use a variant of the \emph{Poincar\'e compactification} method \cite{Poincare} instead.

The previous observations led us to the introduction of the new variables
\begin{equation*} 
u= \frac{p}{\sqrt{p+q}} \,,\qquad v=\frac{2p+2q -nr}{\sqrt{p+q}} \,,\qquad w=\frac{1}{\sqrt{p+q}} \,,
\end{equation*}
where we expect that $w(t)$ tends to zero as $t^{-1}$, and that $u(t)$ and $v(t)$ converge to nontrivial limits. Since this coordinate change produces 
a singularity at $w=0$, we also change the time variable by $\tau = \int_0^t ds/w(s)$.
In terms of the new variables system \eqref{eq_1} becomes
\begin{eqnarray}
  \frac{du}{d\tau} &=&  (\kappa_1 w - \kappa_3 u)(u-v) + \kappa_-(1-uw)\left( 1 - \frac{n(n-1)uw}{(n-2)(2-vw)}\right) - (\kappa_2+\kappa_{-1})uw \nonumber\\
             &&    - uw^2A(u,v,w) \,,\nonumber\\
  \frac{dv}{d\tau} &=& w\left(2\kappa_1(u-v) - (2\kappa_{-1} + n\kappa_2)u + \kappa_- (1-uw)n \frac{2u-nv}{(n-2)(2-vw)}\right) - vw^2A(u,v,w) \,,\label{syst-uvw}\\
  \frac{dw}{d\tau} &=& - w^3 A(u,v,w) \,,\quad A(u,v,w) := \frac{1}{2}\left(\kappa_1(u-v) - \kappa_{-1}u - \kappa_-(1-uw)\frac{n(n-1)u}{(n-2)(2-vw)}\right) \,.
    \nonumber
\end{eqnarray}
Our goal is to prove that solutions converge to a steady state $(u^*,v^*,w^*)$ with $w^*=0$, which obviously has to satisfy 
$- \kappa_3 u^*(u^*-v^*) + \kappa_-=0$, implying 
\begin{equation}\label{def-U}
u^* = U(v^*) := \frac{1}{2} \left(  v^*  +
   \sqrt{(v^*)^2 + 4\kappa_-/\kappa_3}\right) \,,
\end{equation}
since we need $u^*>0$. We intend to show that $v^*$ is determined from the requirement that the large parenthesis in the $v$-equation vanishes.
The argument is essentially that for small values of $w$, the variable $v$ evolves much faster than $w$.

In order to make the slow-fast structure of this system more apparent and to allow the application of basic results from singular perturbation theory,
we assume that the initial value for $w$ is small and define $\ep := (p_0+q_0)^{-1/2}\ll 1$ and the rescaled variable $W= w/\ep$, leading to
\begin{eqnarray}
  \frac{du}{d\tau} &=& - \kappa_3 u(u-v) + \kappa_- + O(\ep) \,,\nonumber\\
  \frac{dv}{d\tau} &=& \ep W\left(2\kappa_1(u-v) - (2\kappa_{-1} + n\kappa_2)u + \kappa_- n \frac{2u-nv}{2(n-2)}\right) +O(\ep^2) \,,\label{syst-uvW}\\
  \frac{dW}{d\tau} &=& - \ep^2 W^3 A(u,v,0) + O(\ep^3)  \,.\nonumber
\end{eqnarray}
The initial data are denoted by 
$$
   u(0)=u_0 :=  \frac{p_0}{\sqrt{p_0+q_0}} >0\,,\qquad v(0)=v_0 := \frac{2p_0+2q_0 -nr_0}{\sqrt{p_0+q_0}}\,,\qquad   W(0)=1\,,
$$
where in the following we consider $u_0$ and $v_0$ as fixed when $\ep\to 0$.
This is a singular perturbation problem in standard form, where $\tau$ plays the role of an initial layer variable. We pass to the limit $\ep\to 0$ 
to obtain the initial layer problem
\begin{eqnarray}
  \frac{d\hat u}{d\tau} &=& - \kappa_3 \hat u(\hat u-\hat v) + \kappa_-   \,,\label{u-layer}\\
  \frac{d\hat v}{d\tau} &=& \frac{d\hat W}{d\tau} = 0 \,,\nonumber
\end{eqnarray}
subject to the initial conditions. By the qualitative behaviour of the right hand side of the first equation, the solution satisfies $\hat v(\tau)=v_0$, 
$\hat W(\tau)=1$, and
$$
   \lim_{\tau\to\infty} \hat u(\tau) = U(v_0)  \,,
$$
with exponential convergence, where $U$ has been defined in \eqref{def-U}. The equation $u=U(v)$ defines the so called \emph{reduced manifold}. 
Since it is exponentially attracting, the 
Tikhonov theorem \cite{Tikhonov} (or rather its extension \cite{Fenichel}) implies that, after the initial layer, i.e. when written in terms of the slow variable 
$\sigma = \ep\tau$, the solution trajectory
remains exponentially close to the \emph{slow manifold}, which is approximated by the reduced manifold, and the flow on the slow manifold satisfies
\begin{eqnarray}
  \frac{dv}{d\sigma} &=& W\left(2\kappa_1(U(v)-v) - (2\kappa_{-1} + n\kappa_2)U(v) + \kappa_- n \frac{2U(v)-nv}{2(n-2)}\right) +O(\ep) \,,\nonumber\\
  \frac{dW}{d\sigma} &=& - \ep W^3 A(U(v),v,0) + O(\ep^2)  \,,\label{syst-vW}
\end{eqnarray}
with $v(0)=v_0$, $W(0)=1$. This is again a singular perturbation problem in standard form, where now $\sigma$ is the initial layer variable.
We repeat the above procedure and consider the limiting layer problem
\begin{eqnarray}
  \frac{d\tilde v}{d\sigma} &=& \tilde W\left(2\kappa_1(U(y\tilde v)-\tilde v) - (2\kappa_{-1} + n\kappa_2)U(\tilde v) 
  + \kappa_- n \frac{2U(\tilde v)-n\tilde v}{2(n-2)}\right)  \,,\label{v-layer}\\
  \frac{d\tilde W}{d\sigma} &=& 0  \,.\nonumber
\end{eqnarray}
The observations
$$
   U(-\infty) = 0 \,,\qquad U(\infty)=\infty \,,\qquad 0< U'(v) < 1 \,,
$$
suffice to show that the right hand side of the first equation is a strictly decreasing function of $v$ with a unique zero $v^*$, which can actually be 
computed explicitly:
$$
  v^* = B\left(\kappa_1 - \kappa_{-1} - \frac{n}{2}\kappa_2 + \frac{n}{2(n-2)}\kappa_-\right)
$$ $$
  \mbox{with}\quad B = 2\sqrt{\frac{\kappa_-}{\kappa_3}} \left( \frac{n^3}{4(n-2)}\kappa_-^2 + 4\kappa_1 \kappa_{-1} + 2n \kappa_1\kappa_2
  + n\kappa_1 \kappa_- + \frac{n^2}{n-2} \kappa_{-1}\kappa_- + \frac{n^3}{2(n-2)\kappa_2 \kappa_-}\right)^{-1/2}
$$
The solution of \eqref{v-layer} with $\tilde v(0)=v_0$ satisfies $\lim_{\sigma\to\infty} \tilde v(\sigma) = v^*$ with exponential convergence. Another application 
of the Tikhonov theorem shows that the slowest part of the dynamics with $t=O(\ep^{-1})$ can be approximated by
\begin{equation}\label{slow-model}
  \frac{dW}{d\sigma} = - \ep W^3 A^* \,,\qquad W(0)=1 \,,
\end{equation}
with 
\begin{equation}\label{A-star}
  A^* := A(U(v^*),v^*,0) = \frac{nB}{16(n-2)^2}(4(n-2)^2\kappa_1\kappa_2 - 2n(n-2)\kappa_-(\kappa_1+\kappa_{-1}) - n^2(n-1)\kappa_-^2) >0 \,,
\end{equation}
by \eqref{alpha-smaller-0}. This gives the approximation
$$
   W(\sigma) = (1+2 A^*\ep\sigma)^{-1/2} \,.
$$
The results of \cite{Fenichel} imply that the approximations are accurate with errors of order $\ep$ uniformly with respect to time.

\begin{thm}\label{thm:sing-pert}
Let \eqref{alpha-smaller-0} hold. Then, for $\ep>0$ small enough, the solution of \eqref{syst-uvw} with initial conditions
$$
    u(0) = u_0>0 \,,\qquad v(0) = v_0 \in\mathbb{R} \,, \qquad w(0)=\ep \,,
$$
satisfies
\begin{eqnarray*}
   u(\tau) &=& \hat u(\tau) - U(v_0) + U(\tilde v(\ep\tau)) + O(\ep) \,,\\
   v(\tau) &=& \tilde v(\ep\tau) + O(\ep) \,,\\
   w(\tau) &=& \ep(1 + 2A^* \ep^2\tau)^{-1/2} + O(\ep^2) \,,
\end{eqnarray*}
uniformly in $\tau\ge 0$, where $U$ is given in \eqref{def-U}, $\hat u$ solves \eqref{u-layer}, $\tilde v$ solves \eqref{v-layer}, and $A^*$ is given in
\eqref{A-star}.
\end{thm}

Actually more can be deduced. In terms of the original time variable $t$, the equation for $w$ in \eqref{syst-uvw} becomes
\begin{equation}\label{dot-w}
   \dot w = -w^2 A(u,v,w) \,.
\end{equation}
Under the assumptions of Theorem \ref{thm:sing-pert}, $A(u,v,w)$ is uniformly close to the positive constant $A^*$ and therefore uniformly positive
for large enough $t$.
This implies that $w$ tends to zero as $t\to\infty$. The slow manifold of the system \eqref{syst-vW} reduces to the steady state $(v,W)=(v^*,0)$ for
$W=0$. Therefore $v$ tends to $v^*$ as $t\to\infty$. Analogously, the slow manifold of \eqref{syst-uvW} reduces to the steady state $(u,v,W)=
(u^*=U(v^*),v^*,0)$ at $W=0$, implying convergence of $u$ to $u^*$. This in turn implies convergence of $A(u,v,w)$ to $A^*$, which can be used
in \eqref{dot-w}.

\begin{cor}\label{cor:1}
Let the assumptions of Theorem \ref{thm:sing-pert} hold. Then
$$
  \lim_{t\to\infty} u(t) = u^* \,,\qquad \lim_{t\to\infty} v(t) = v^* \,,\qquad w(t) = \frac{1}{A^* t} + O\left( \frac{1}{t^2}\right) \quad\mbox{as } t\to\infty \,.
$$
\end{cor}

Finally, we reformulate these results in terms of the original variables, verifying the formal asymptotics of \cite{firstarticle} for initial data, which
are in a sense already 'close enough' to the polynomially growing solutions.
\begin{thm}
Let \eqref{alpha-smaller-0} hold, let $c_2\ge c_1>0$, and let $\delta>0$ be small enough. Let the initial data satisfy
$$
   p_0 = \frac{c_1}{\delta} \,,\qquad q_0 = \frac{1}{\delta^2} \,,\qquad r_0 = \frac{2}{n\delta^2} + \frac{c_2}{n\delta}
$$
Then the solution of \eqref{eq_1} with $(p(0),q(0),r(0)) = (p_0,q_0,r_0)$ satisfies
$$
   p(t) = u^*A^* t + o(t)\,,\qquad q(t) = (A^*)^2 t^2 + o(t^2) \,,\qquad r(t) = \frac{2}{n}(A^*)^2 t^2 + o(t^2) \,,\qquad\mbox{as } t\to\infty \,.
$$
\end{thm}
\begin{proof}
We just need to verify that the assumptions of this theorem imply the assumptions of Theorem \ref{thm:sing-pert}. The result is then a direct consequence 
of Corollary \ref{cor:1}. Actually the assumptions of Theorem \ref{thm:sing-pert} hold with $\ep\approx\delta$, since
$$
   u_0 = \frac{c_1}{\sqrt{1+c_1\delta}} \,,\qquad v_0 = \frac{2c_1-c_2}{\sqrt{1+c_1\delta}} \,,\qquad w_0 = \frac{\delta}{\sqrt{1+c_1\delta}} \,.
$$
\end{proof}

\section{Discussion}
In this work a mathematical model for aggregation via cross-linking has been analyzed. Besides the basic assumption that aggregating {\em particles} (here
p62 oligomers) need to have at least $n=3$ binding sites for {\em cross-linkers} (here ubiqutinated cargo), the rate constants for binding reactions need to be
large enough compared to those for the unbinding reactions (the opposite of inequality \eqref{alpha-bigger-1}) for stable aggregates to exist. Under a 
stronger condition (inequality \eqref{alpha-smaller-0}) aggregates grow indefinitely in the presence of an unlimited supply of free particles and cross-linkers.
These conjectures from \cite{firstarticle}, where the model has been formulated, have been partially proven in this work. It has been shown in Section
\ref{sec:zero} that small aggregates get completely degraded under the condition \eqref{alpha-bigger-1} and that they grow under the opposite condition.
In the latter case, but when \eqref{alpha-smaller-0} does not hold, there exists an equilibrium configuration with positive aggregate size. Finally, it has been 
shown in Section \ref{sec:poly} that under the condition \eqref{alpha-smaller-0} aggregate size grows polynomially with time (actually like $t^2$) for 
appropriate initial states.

\begin{figure}[h!] \centering
\includegraphics[width=0.5\textwidth]{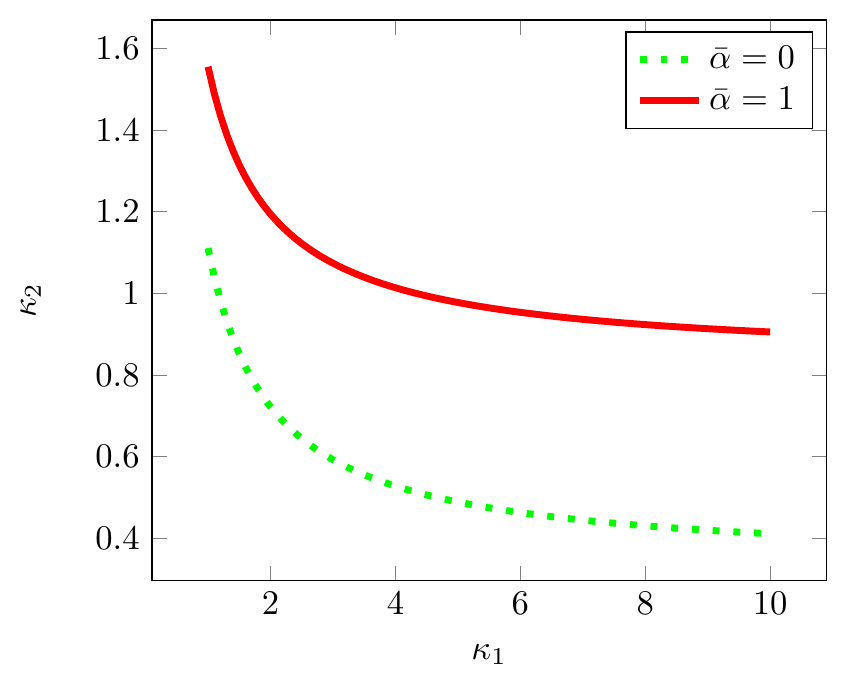}
\caption{Bifurcation diagram obtained for $\kappa_{-1}=\kappa_{-}=1$}
\label{fig:7} 
\end{figure}

The constants $\kappa_1$, $\kappa_2$ in the model have to be interpreted as the products of rate constants with the concentrations of free cross-linkers 
and, respectively, of free particles. This means that the conditions \eqref{alpha-bigger-1} and \eqref{alpha-smaller-0} are actually conditions for these
concentrations. Fig. \ref{fig:7} shows a bifurcation diagram in terms of $\kappa_1$ and $\kappa_2$ with the curves $\bar\alpha=1$, corresponding to
equality in \eqref{alpha-bigger-1}, and $\bar\alpha=0$,  corresponding to equality in \eqref{alpha-smaller-0}. The qualitative behaviour is no surprise:
Close to the origin, i.e. for small concentrations of free particles and cross-linkers, aggregates are unstable. Moving to the right and/or up we pass through
two bifurcations to stable finite aggregate size and, subsequently, to polynomial growth of aggregates. Less obvious is the fact that the picture is rather 
unsymmetric with respect to the two parameters. The condition $(n-2)\kappa_2> \kappa_-$ is necessary for the existence of stable aggregates, regardless 
of the value of $\kappa_1$, whereas arbitrarily small values of $\kappa_1$ can be compensated by large enough $\kappa_2$. This means that, if the concentration of free particles is below a threshold, even a large concentration of cross-linkers does not lead to aggregation, whereas arbitrarily small 
numbers of cross-linkers are used for aggregation if the particle concentration is high. For the application in cellular autophagy this means that aggregation
will only happen for large enough concentrations of p62 oligomers. However, arbitrarily small amounts of ubiquitinated cargo can be aggregated in the
presence of a large enough supply of oligomers.

This work has been motivated by the experimental results of \cite{Martens}, where aggregates have been detected by light microscopy. If the evolution of
single aggregates can be followed, the growth like $t^2$ might be observed as a fluorescence signal of tagged oligomers, which goes like $t^2$, or 
cross section areas going like $t^{4/3}$, if a roughly spherical shape of aggregates is assumed. For quantitative predictions of such experiments, the 
model should be extended in various ways. First, the {\em limited supply of free p62 oligomers and of free cross-linkers} should be taken into account. This is
straightforward for the modeling of a single aggregate, but if many aggregates develop simultaneously, they will compete for the free particles. Apart from
that the number of aggregates has to be predicted, which requires modeling of the {\em nucleation} process. Finally, it is very likely that the {\em coagulation} 
of aggregates plays an important role. A growth-coagulation model for distributions of aggregates, based on the growth model \eqref{eq_1} would be
prohibitively complex. It is therefore the subject of ongoing work to formulate, analyze, and simulate a growth-coagulation model based on the multiscale
analysis of Section \ref{sec:poly}, where aggregates are only described by the size parameter $r$ (number of p62 oligomers in the aggregate), whose
evolution is determined by the slow dynamics \eqref{slow-model}, which translates to an equation of the form $\dot r = C\sqrt{r}$ for $r$. This approach raises
several challenging issues such as the development of an efficient simulation algorithm or the existence and stability of equilibrium aggregate distributions.

\bibliography{mybib}{}

\begin{thebibliography}{1}

\bibitem{firstarticle}
J.~Delacour, M.~Doumic, S.~Martens, C.~Schmeiser, and G.~Zaffagnini.
\newblock A mathematical model of p62-ubiquitin aggregates in autophagy.
\newblock arXiv:2004.07926.

\bibitem{dumortier}
F.~Dumortier.
\newblock Techniques in the theory of local bifurcations: Blow-up, normal
  forms, nilpotent bifurcations, singular perturbations.
\newblock In {\em Bifurcations and Periodic Orbits of Vector Fields}, pages
  19--73. Kluwer Acad. Publ., 1993.

\bibitem{Fenichel}
N.~Fenichel.
\newblock Geometric singular perturbation theory for ordinary differential
  equations.
\newblock {\em Journal of Differential Equations}, 31:53--98, 1979.

\bibitem{Poincare}
H.~Poincar\'e.
\newblock M\'emoire sur les courbes d\'efinies par une \'equation
  diff\'erentielle (i).
\newblock {\em Journal de math\'ematiques pures et appliqu\'ees 3$^e$ s\'erie,
  tome 7}, pages 375--422, 1881.

\bibitem{Schmeiser}
C.~Schmeiser and R.~Wei\ss.
\newblock Asymptotic analysis of singular singularly perturbed boundary value
  problems.
\newblock {\em SIAM J. Math. Anal.}, 17:560--579, 1986.

\bibitem{Szmolyan}
P.~Szmolyan.
\newblock Transversal heteroclinic and homoclinic orbits in singular
  perturbation problems.
\newblock {\em Journal of Differential Equations}, 92:252--281, 1991.

\bibitem{Tikhonov}
A.~N. Tikhonov.
\newblock Systems of differential equations containing small parameters in the
  derivatives.
\newblock {\em Matematicheskii sbornik}, 73:575--586, 1952.

\bibitem{Martens}
G.~Zaffagnini, A.~Savova, A.~Danieli, J.~Romanov, S.~Tremel, M.~Ebner,
  T.~Peterbauer, M.~Sztacho, R.~Trapannone, A.K. Tarafder, C.~Sachse, and
  S.~Martens.
\newblock p62 filaments capture and present ubiquitinated cargos for autophagy.
\newblock {\em The EMBO Journal}, 37(5):e98308, 03 2018.

\end{thebibliography}
\bibliographystyle{plain}
\end{document}